\documentclass[11pt]{amsart}
\usepackage{amssymb}
\usepackage{amsfonts}

\usepackage{bbm}
\usepackage{mathrsfs}
\usepackage{color}

\usepackage{array}

\usepackage{amsfonts,amsmath, amssymb}

\newcommand{\bea}{\begin{eqnarray}}

\newcommand{\eea}{\end{eqnarray}}

\newcommand{\be}{\begin {equation}}

\newcommand{\ee}{\end{equation}}

\newtheorem{theorem}{Theorem}[section]

\newtheorem{corollary}[theorem]{Corollary}
\newtheorem{lemma}[theorem]{Lemma}
\newtheorem{condition}[theorem]{Condition}

\newtheorem{remark}[theorem]{Remark}

\newtheorem{claim}[theorem]{Claim}

\newtheorem{convention}[theorem]{Convention}

\begin{document}

\title{continued fraction expansions of algebraic numbers}

\author {Xianzu Lin }

\date{ }
\maketitle
   {\small \it College of Mathematics and Computer Science, Fujian Normal University, }\\
    \   {\small \it Fuzhou, {\rm 350108}, China;}\\
      \              {\small \it Email: linxianzu@126.com}
\begin{abstract}
In this paper we establish properties of independence for the
continued fraction expansions of two algebraic numbers. Roughly
speaking, if the continued fraction expansions of two irrational
real algebraic numbers have the same long sub-word, then the two
continued fraction expansions have the same tails. If the two
expansions have mirror symmetry long sub-words, then both the two
algebraic numbers are quadratic. Applying the above results, we
prove a theorem analogous to the Roth's theorem about
approximation by algebraic numbers.

\end{abstract}






\section{Introduction}
It is a well-known fact that every quadratic irrational real
number can be represented by an eventually periodic continued
fraction. By contrast, no analogous results are known for
algebraic numbers of higher degree. In fact we can not write down
explicitly  the continued fraction expansion  of a single real
algebraic number of degree higher than 2, and we do not know
whether the partial quotients of such expansions are bounded or
unbounded.

In the past 10 years, some breakthroughs have been obtained in
this direction by Adamczewski, Bugeaud and other people
\cite{ab2,ab3,ab4,abd,bu}, highlighted in \cite{bu}. Before
introducing the main result in \cite{bu}, we need some
preparations.

We say that an infinite word $\textbf{a}= a_1a_2\cdots$ of
elements from an alphabet $\Omega$ has $long \ repetition$ if it
satisfies (\romannumeral1), (\romannumeral2) and (\romannumeral3)
of Condition \ref{ccc}, where the length of a finite word $A$ is
denoted by $|A |$, and the mirror image $a_na_{n-1}\cdots a_1$ of
a finite word $B=a_1a_2\cdots a_n$ is denoted by $\overline{B}$.
We say $\textbf{a}= a_1a_2\cdots$ has $long \ mirror \ repetition$
if it satisfies (\romannumeral1'), (\romannumeral2) and
(\romannumeral3) of Condition \ref{ccc}
 \begin{condition} \label{ccc}
There exist three sequences of finite nonempty words
$\{A_n\}_{n\geq1}$, $\{A'_n\}_{n\geq1}$, $\{B_n\}_{n\geq1}$ such
that:
   \begin{enumerate}
\item [(i)] for any $n\geq1$, $A_nB_nA'_nB_n$ is a prefix of
$\textbf{a}$; \item [(i')] for any $n\geq1$,
$A_nB_nA'_n\overline{B_n}$ is a prefix of $\textbf{a}$;
 \item
[(ii)]the sequence $\{|B_n|\}_{n\geq1}$ is strictly increasing;
 \item
[(iii)]there exists a positive constant $L$ such that $$(|
A_n|+|A'_n|)/|B_n|\leq L,$$ for every $n\geq1$.
\end{enumerate}
\end{condition}

The main results in \cite{bu} are the following two theorems:

 \begin{theorem} \label{main1}
Let $$\alpha=[[\alpha];a_1,a_2,\cdots]$$ be  the continued
fraction expansions of a real algebraic number of degree higher
than 2, and let $\{\frac{p_{n}}{q_{n}}\}_{n\geq0}$  be the
sequence of convergents. Assume that the sequence
$\{(q_{n})^{1/n}\}_{n\geq0}$ is bounded. Then, the infinite word
$a_1a_2\cdots$ has no long repetitions.
\end{theorem}

 \begin{theorem} \label{main2}
Let assumptions be as above.  Then, the infinite word
$a_1a_2\cdots$ has no long mirror repetitions.
\end{theorem}

For an infinite word $\textbf{a}= a_1a_2\cdots$ of elements from
an alphabet $\Omega$ and a positive integer $n$, set
$$p(\textbf{a},n):=Card\{a_{i+1}\cdots a_{i+n}| i\geq0\}.$$ Then
$p(\textbf{a},n)$ is the number of distinct blocks of $n$
consecutive letters occurring in $\textbf{a}$.
 Theorem \ref{main1} implies that
$$\lim_{n\rightarrow +\infty} \frac{p(\textbf{a},n)}{n}=+\infty,$$ where $\textbf{a}$ is the continued fraction expansion of an algebraic number of degree
higher than 2. This result combining with a fundamental property
about automatic sequences (cf.\cite{co}) immediately implies that
the continued fraction expansion of an algebraic number of degree
higher than 2 can not be generated by a finite automaton.

Theorem \ref{main1}  and its corollaries are very similar to the
corresponding results  about expansions of algebraic numbers to
integer bases \cite{ab}. In \cite{ab1}, Adamczewski and Bugeaud
further explored the independence of $b$-ary expansions of two
irrational real algebraic numbers $\alpha$ and $\beta$.

Let $\textbf{a}= a_1a_2\cdots$ and $\textbf{a}'= a'_1a'_2\cdots$
be two infinite words of elements from an alphabet $\Omega$. The
following is a condition about the pair
$(\textbf{a},\textbf{a}')$:

 \begin{condition} \label{c1}
There exist three sequences of finite nonempty words
$\{A_n\}_{n\geq1}$, $\{A'_n\}_{n\geq1}$, $\{B_n\}_{n\geq1}$ such
that:
   \begin{enumerate}
\item [(i)] for any $n\geq1$, the word $A_nB_n$ is a prefix of the
word $\textbf{a}$ and the word $A'_nB_n$ is a prefix of the word
$\textbf{a}'$; \item [(i')] for any $n\geq1$, the word $A_nB_n$ is
a prefix of the word $\textbf{a}$ and the word
$A'_n\overline{B_n}$ is a prefix of the word $\textbf{a}'$;
 \item
[(ii)]the sequence $\{|B_n|\}_{n\geq1}$ tends to infinity; \item
[(iii)]there exists a positive constant $L$ such that $$(|
A_n|+|A'_n|)/|B_n|\leq L,$$ for each $n\geq1$.
\end{enumerate}
\end{condition}

The main result in \cite{ab1} is:
 \begin{theorem} \label{main3}
Let $b\geq2$ be a fixed integer. Let $\alpha$ and $\alpha'$ be two
irrational real algebraic numbers. If their $b$-ary expansions
$$\alpha=[\alpha]+0.a_1a_2\cdots,$$ and
$$\alpha'=[\alpha]+0.a'_1a'_2\cdots$$ satisfy (i), (ii)and (iii) of Condition \ref{c1},
then the two infinite words $\textrm{a}= a_1a_2\cdots$  and
$\textrm{a}'= a'_1a'_2\cdots$ have the same tail.
\end{theorem}

In this paper, we show that similar results of independence hold
for the continued fraction expansions of two algebraic numbers.
Our main results are:

 \begin{theorem} \label{main4}
Let $$\alpha=[[\alpha];a_1,a_2,\cdots],$$ and
$$\alpha'=[[\alpha'];a'_1,a'_2,\cdots]$$ be  the continued fraction expansions of two
irrational real algebraic numbers, and let
$\{\frac{p_{n}}{q_{n}}\}_{n\geq0}$ and
$\{\frac{p'_{n}}{q'_{n}}\}_{n\geq0}$ be  respectively the sequence
of convergents of $\alpha$ and $\alpha'$. Assume that the sequence
$\{(q_{n}q'_{n})^{1/n}\}_{n\geq0}$ is bounded. If the two infinite
words $\textbf{a}= a_1a_2\cdots$  and $\textbf{a}'=
a'_1a'_2\cdots$ satisfy $(i),(ii)$ and $(iii)$  of Condition
\ref{c1}, then they have the same tail. Moreover, if
$$\limsup_{n\rightarrow\infty} \mid|A_n|-|A'_n|\mid=+\infty,$$ then
both $\alpha$ and $\alpha'$ are quadratic irrationals
\end{theorem}

 We say that two finite words $$A=a_1a_2\cdots a_n$$
 and $$B=b_1b_2\cdots b_n$$ are $cycle\  mirror\  symmetry$ if there
 exists an positive integer $i\leq n$ such that $$b_nb_{n-1}\cdots b_1=a_i\cdots a_na_1\cdots a_{i-1}.$$
 \begin{theorem} \label{main5}
Let assumptions be as above. If the two infinite words
$\textbf{a}= a_1a_2\cdots$  and $\textbf{a}'= a'_1a'_2\cdots$
satisfy (i'), (ii) and (iii)  of Condition \ref{c1},  then both
$\alpha$ and $\alpha'$ are quadratic irrationals. Moreover the
shortest periods of the $\textbf{a}$  and $\textbf{a}'$ are cycle
mirror symmetry.
\end{theorem}
\begin{remark}\label{remar}
The proofs below   show that Theorems \ref{main4} and \ref{main5}
are still valid if we replace   (iii)  of Condition \ref{c1} and
the boundness of $\{(q_{n}q'_{n})^{1/n}\}_{n\geq0}$ with the
following condition
 \begin{condition} \label{cfggh1}
There exist positive numbers $\delta$ and $L$ such that
$$(q_{k_{n}}q'_{l_{n}})^{1+\delta}< L q_{k_{n}+m_n}q'_{l_{n}+m_n}$$
for each $n\geq1$, where $k_n=|A_n|$, $l_n=|A'_n|$, and
$m_n=|B_n|$.
\end{condition}
\end{remark}

Applying Theorems \ref{main4} and \ref{main5} to the case
$\alpha=\alpha'$, we recover Theorems \ref{main1} and \ref{main2}
immediately.

Another main result of this paper is Diophantine approximation by
algebraic numbers. The classical theory of Diophantine
approximation of reals by rationals has the geometric
interpretation of approximating elements of the boundary of the
hyperbolic plane by the orbit of infinity under the modular group.
For example, the celebrated Roth's Theorem \cite{ro} can be stated
as:
 \begin{theorem} [Roth]\label{main99w} Let $\epsilon $ be a positive number
and let $\xi$ be an irrational real number. If there exist
infinitely many  \[\begin{matrix}A=\begin{pmatrix}a&b\\
c&d\end{pmatrix}\in PSL(2,\mathbb{Z})
\\
\end{matrix}\] such that
$$|\xi-\tfrac{a\infty+b}{c\infty+d}|<\|A\|^{-2-\epsilon},$$
then $\xi $ is transcendental.
\end{theorem}
In this paper we will show that an analogy to Roth's theorem holds
when the point of infinity is replaced by an irrational real
algebraic number. Let $ PSL(2,\mathbb{Z})$ be the projective
linear group
 of $2\times2$ matrices with integer coefficients and unit
determinant. For any
\[\begin{matrix}A=\begin{pmatrix}a&b\\
c&d\end{pmatrix}\in PSL(2,\mathbb{Z}),
\\
\end{matrix}\]
Set $\|A\|=\max(|c|,|d|)$. For an irrational real number
$\alpha$, an element \[\begin{matrix}A=\begin{pmatrix}a&b\\
c&d\end{pmatrix}\in PSL(2,\mathbb{Z}),
\\
\end{matrix}\] acts on $\alpha$ by
 \[\begin{matrix}A=\begin{pmatrix}a&b\\
c&d\end{pmatrix}\alpha=\frac{a\alpha+b}{c\alpha+d}.
\\
\end{matrix}\]

Let $\Theta_{\alpha}=PSL(2,\mathbb{Z})\alpha$ be the orbit of
$\alpha$ for the action of $ PSL(2,\mathbb{Z})$. When $\alpha$ is
of degree higher than 2,  for any $\beta\in\Theta_{\alpha}$, set
$$\|\beta\|=\|A\|,$$ where $A$ is the
unique element of $ PSL(2,\mathbb{Z})$ such that $\beta=A\alpha$.

\begin{remark}\label{rem}
The definition of the norm $\|\cdot\|$ above depends upon
$\alpha$. But it is easy to see that when
$\alpha\neq\alpha'\in\Theta_{\alpha}$, there exist two positive
constants $c_1$ and $c_2$ such that the corresponding norms
$\|\cdot\|_{\alpha}$ and $\|\cdot\|_{\alpha'}$ satisfy
$$c_1\|\cdot\|_{\alpha'}<\|\cdot\|_{\alpha}<c_2\|\cdot\|_{\alpha'}.$$
\end{remark}

 \begin{theorem} \label{main99}
Let $\alpha$ be an irrational real algebraic number  of degree
$d>2$. Let $\epsilon $ be a positive number and let $\xi$ be an
irrational real number not in $\Theta_{\alpha}$. If there exist
infinitely many $\beta\in\Theta_{\alpha}$ such
$$|\xi-\beta|<\|\beta\|^{-2-\epsilon},$$ then $\xi $ is
transcendental.
\end{theorem}

Similar result holds for quadratic  real number. Let $\alpha$ be a
fixed quadratic irrational  real number, and let $x^{\sigma}$ be
the conjugate of any quadratic number $x$.  For any
$\beta=\frac{a\alpha+b}{c\alpha+d}\in\Theta_{\alpha}$, set
$$\|\beta\|=|
\tfrac{1}{\beta-\beta^{\sigma}}|=|\tfrac{(c\alpha+d)(c\alpha^{\sigma}+d)}{\alpha-\alpha^{\sigma}}|.$$
The proof of Theorem \ref{main99} also implies:
 \begin{theorem} \label{main990}
Let $\epsilon $ be a positive number and let $\xi$ be an
irrational number not in $\Theta_{\alpha}$. If there exist
infinitely many $\beta\in\Theta_{\alpha}$ such
$$|\xi-\beta|<\|\beta\|^{-1-\epsilon},$$ then $\xi $ is
transcendental.
\end{theorem}
This is essentially Theorem 4.4 from \cite{bu1}.

This paper is structured as follows: In Section 2, we give
preliminaries that will be used throughout this paper. In Sections
3 and 4, we give proofs of Theorems \ref{main4} and \ref{main5}.
In Section 5,  we give proofs of
 Theorems \ref{main99} and \ref{main990}.

 \section{preliminaries}
In this paper, we write $$[a_0;a_1,a_2,\cdots,a_n]$$ for the
finite continued fraction expansion
$$a_0+\frac{1}{a_1+\frac{1}{a_2+\cdots+\frac{1}{a_n}}},$$ and write $$[a_0;a_1,a_2,\cdots,a_n,\cdots]$$ for the
infinite continued fraction expansion
$$a_0+\frac{1}{a_1+\frac{1}{a_2+\cdots+\frac{1}{a_n+\cdots}}},$$
where $a_1,a_2,\cdots, $ are positive integers and $a_0$ is an
integer. An eventually periodic continued fraction is written as
$$[a_0;a_1,a_2,\cdots,a_{k-1},\overline{a_k,\cdots,a_{k+m-1}}],$$
where $a_0;a_1,a_2,\cdots,a_{k-1}$ is the preperiod and
$a_k,\cdots,a_{k+m-1}$ is the shortest period.

The sequence of convergents of
$$a=[a_0;a_1,a_2,\cdots,a_n,\cdots]$$ is defined by
$$p_{-2}=0, \   p_{-1}=1, \   p_{n}=a_np_{n-1}+p_{n-2}  \   (n\geq 0),$$
$$q_{-2}=1, \   q_{-1}=0, \   q_{n}=a_nq_{n-1}+q_{n-2}  \   (n\geq 0).$$
We have (cf.\cite{hw})
 \begin{lemma}\label{ss}
$$[a_0;a_1,a_2,\cdots,a_n]=\frac{p_{n}}{q_{n}},$$
$$|a-\frac{p_{n}}{q_{n}}|<\frac{1}{q_{n}q_{n+1}},$$ and
$$q_{m+n}\geq 2^{\tfrac{m-1}{2}}q_{n},$$ for $m,n\geq 1.$
   \end{lemma}

For any finite nonempty word of integers $B=b_0b_1b_2\cdots b_n $,
set

\[\begin{matrix}
\textsl{M}(B)=\begin{pmatrix}b_{0}&1\\1&0\end{pmatrix}\begin{pmatrix}b_{1}&1\\1&0\end{pmatrix}\cdots\begin{pmatrix}b_{n}&1\\1&0\end{pmatrix}.\\
\end{matrix}\]

Then it is well-known that (cf.\cite{fr})
\[\begin{matrix}\textsl{M}(a_0,a_1,a_2,\cdots,a_n)=\begin{pmatrix}p_{n}&p_{n-1}\\q_{n}&q_{n-1}\end{pmatrix},\\
\end{matrix}\]

Hence
\begin{equation}\label{ss1}p_{n}q_{n-1}-p_{n-1}q_{n}=(-1)^{n+1}.\end{equation}

When $a_0=0$,
 by taking the conjugation, we get \[\begin{matrix}\textsl{M}(0,a_n,a_{n-1},\cdots,a_1)=\begin{pmatrix}q_{n-1}&p_{n-1}\\q_{n}&p_{n}\end{pmatrix}.\\
\end{matrix}\]  Hence
\begin{equation}\label{ss12}\frac{q_{n}}{q_{n-1}}=[a_n;a_{n-1},\cdots,a_1].\end{equation}

   A simple proof by induction shows that:

    \begin{lemma}\label{ss3}
    For two finite words of positive integers  $$a_0a_1a_2\cdots a_m, $$and$$  b_0b_1b_2\cdots
    b_n,
    $$
$$\textsl{M}(a_0,a_1,a_2,\cdots,a_m)=\textsl{M}(b_0,b_1,b_2,\cdots,b_n)$$ implies $$a_0a_1a_2\cdots a_m=b_0b_1b_2\cdots
b_n.
    $$
   \end{lemma}

Another well-known result about continued fractions is that:
    \begin{lemma}\label{xxx}
 Let $\alpha$ and $\beta$ be two irrational real numbers. If there exists an $A\in PSL(2,\mathbb{Z})$ such that $\beta=A\alpha$.
, then the continued fraction expansions of $\alpha$ and $\beta$
have the same tail.
   \end{lemma}

As in \cite{ab2,ab3,ab4,abd,bu} the proofs of Theorems \ref{main4}
and \ref{main5} need the Schmidt subspace theorem \cite{sch}.

 \begin{theorem} \label{main8}
Let $n\geq1$ be an integer. For every
$$\textbf{x}=(x_0,\cdots,x_n)\in\mathbb{Z}^{n+1},$$ set
$$\|\textbf{x}\|=\max_{i}(| x_{i}|).$$
Let $L_{0}(\textbf{x}),\cdots,L_{n}(\textbf{x})$ be linearly
independent linear forms in $n+1$ variables with algebraic
coefficients. Then for any positive number $\epsilon$, the
solutions $\textbf{x}\in\mathbb{Z}^{n+1}$ of the inequality
$$\prod_{i=1}^{n+1}|L_{i}(\textbf{x})|\leq\|\textbf{x}\|^{-\epsilon}$$ lie in finitely many proper linear subspaces of
$\mathbb{Q}^{n+1}$
\end{theorem}

We also need the following lemma which is contained in the proof
of Theorem 3.1 from \cite{bu}. For the readers' convenience, we
give here a different proof.
 \begin{lemma}\label{ssTte}
 Let $\alpha$ be
an real algebraic number of degree higher than 2, and let
$\{p_{n}/q_{n}\}_{n\geq0}$ be the sequence of convergents of the
continued fraction expansion of $\alpha$. Then there do not exist
a nonzero element $(x_{1},x_{2},x_{3},x_{4})\in \mathbb{Q}^{4}$,
and an infinite subset $\mathbb{N}'$ of $\mathbb{N}$ such that

$$x_{1}q_{n-1}+x_{2}p_{n-1}+x_{3}q_{n}+x_{4}p_{n}=0,$$
 for each $n\in \mathbb{N}'$.
   \end{lemma}
\begin{proof}
Assume that there exist a nonzero element
$(x_{1},x_{2},x_{3},x_{4})\in \mathbb{Q}^{4}$, and an infinite
subset $\mathbb{N}'$ of $\mathbb{N}$ such that

\begin{equation}\label{for5}x_{1}q_{n-1}+x_{2}p_{n-1}+x_{3}q_{n}+x_{4}p_{n}=0,\end{equation}
 for each $n\in \mathbb{N}'$. First, we have
 $(x_{1},x_{2})\neq(0,0)$, otherwise dividing (\ref{for5}) by
 $q_{n}$ and letting $n$ tend to infinity along $\mathbb{N}'$
 would implies that $\alpha$ is rational. Similarly, $(x_{3},x_{4})\neq(0,0)$. Now we assume that
 $x_{1}\neq0$ and define three linearly independent linear forms
 $$L_{1}(X,Y,Z)=(1+\alpha)\frac{x_{2}}{x_{1}}X+\alpha\frac{x_{3}}{x_{1}}Y+\alpha\frac{x_{4}}{x_{1}}Z,$$
 $$L_{2}(X,Y,Z)=Z-\alpha Y,\   L_{3}(X,Y,Z)=X.$$ By (\ref{for5}) and
 Lemma \ref{ss}, we have
 $$\prod_{1\leq i\leq3}|L_{i}(p_{n-1},q_{n},p_{n})|<p_{n-1}q^{-2}_{n}<(|\alpha|+2)q^{-1}_{n},$$
 for each $n\in \mathbb{N}'$. It follows from Theorem \ref{main8}
 that there exist a nonzero element
$(y_{1},y_{2},y_{3})\in \mathbb{Q}^{3}$, and an infinite subset
$\mathbb{N}''$ of $\mathbb{N}'$ such that
\begin{equation}\label{for6}y_{1}p_{n-1}+y_{2}q_{n}+y_{3}p_{n}=0,\end{equation}
 for each $n\in \mathbb{N}''$. We observe as before that $
 y_{1}\neq0$. Now combining (\ref{for5}) and (\ref{for6}) implies
 that there exists a nonzero element
$(a,b,c,d)\in \mathbb{Q}^{4}$ such that $$p_{n-1}=aq_{n}+bp_{n}$$
and $$q_{n-1}=cq_{n}+dp_{n}$$ for each $n\in \mathbb{N}''$.
Letting $n$ tend to infinity along $\mathbb{N}''$, we get
\begin{equation}\label{for7}\alpha=\frac{a+b\alpha}{c+d\alpha}.\end{equation} As $\alpha$ is an algebraic number of degree higher than
2, (\ref{for7}) forces $$a=d=0, \ b=c\neq0,$$ and then
$$p_{n-1}q_{n}=p_{n}q_{n-1},$$  for each $n\in \mathbb{N}''$. This contradicts (\ref{ss1}).
The case $x_{2}\neq0$ can be treated similarly.
\end{proof}

 \section{Proof of Theorem \ref{main4}}

\begin{proof}[Proof of Theorem \ref{main4}]
Assume that (i), (ii) and (iii)  of Condition \ref{c1} are
satisfied with three sequences of finite nonempty words
$\{A_n\}_{n\geq1}$, $\{A'_n\}_{n\geq1}$, $\{B_n\}_{n\geq1}$. If
$A_n=Ca$ and $A'_n=C'a$ for some $n$ and some positive integer
$a$, we can replace $A_n,A'_n,B_n$ with $C,C',aB_n$ without
violating (i), (ii) and (iii)  of Condition \ref{c1}. Hence we can
 further require that

 \begin{convention} \label{convi1}
 the last letter of $A_n$ and the last
letter of $A'_n$ are different if $\min(|A_n|,|A'_n|)\geq3$.
\end{convention}

 Set
$k_n=|A_n|$, $l_n=|A'_n|$, and $m_n=|B_n|$. Then we have

\[\begin{matrix}\begin{pmatrix}p_{k_{n}}&p_{k_{n}-1}\\q_{k_{n}}&q_{k_{n}-1}\end{pmatrix}\textsl{M}(B_n)=
\begin{pmatrix}p_{k_{n}+m_n}&p_{k_{n}+m_n-1}\\q_{k_{n}+m_n}&q_{k_{n}+m_n-1}\end{pmatrix},\\
\end{matrix}\]

and

\[\begin{matrix}\begin{pmatrix}p'_{l_{n}}&p'_{l_{n}-1}\\q'_{l_{n}}&q'_{l_{n}-1}\end{pmatrix}\textsl{M}(B_n)=
\begin{pmatrix}p'_{l_{n}+m_n}&p'_{l_{n}+m_n-1}\\q'_{l_{n}+m_n}&q'_{l_{n}+m_n-1}\end{pmatrix}.\\
\end{matrix}\]
\bigskip
The above two identities immediately imply
{\setlength{\arraycolsep}{0pt}
\begin{eqnarray} \label{for77}
&&\begin{pmatrix}p_{k_{n}}&p_{k_{n}-1}\\q_{k_{n}}&q_{k_{n}-1}\end{pmatrix}
\begin{pmatrix}q'_{l_{n}-1}&-p'_{l_{n}-1}\\-q'_{l_{n}}&p'_{l_{n}}\end{pmatrix} \\
&=&\begin{pmatrix}p_{k_{n}+m_n}&p_{k_{n}+m_n-1}\\q_{k_{n}+m_n}&q_{k_{n}+m_n-1}\end{pmatrix}
\begin{pmatrix}q'_{l_{n}+m_n-1}&-p'_{l_{n}+m_n-1}\\-q'_{l_{n}+m_n}&p'_{l_{n}+m_n}\end{pmatrix}.\nonumber
\end{eqnarray}
}

We define four linearly independent linear forms as follows:
{\setlength{\arraycolsep}{0pt}
\begin{eqnarray*}
L_{1}(X_1,X_2,X_3,X_4)&=&\alpha \alpha'X_1-\alpha X_2-\alpha'X_3+X_4,\\
L_{2}(X_1,X_2,X_3,X_4)&=&\alpha'X_1-X_2,\\
L_{3}(X_1,X_2,X_3,X_4)&=&\alpha X_1-X_3,\\
L_{4}(X_1,X_2,X_3,X_4)&=&X_1.\\
\end{eqnarray*}
} Set
$$\phi_n=(q_{k_{n}}q'_{l_{n}-1}-q_{k_{n}-1}q'_{l_{n}},q_{k_{n}}p'_{l_{n}-1}-q_{k_{n}-1}p'_{l_{n}},
p_{k_{n}}q'_{l_{n}-1}-p_{k_{n}-1}q'_{l_{n}},p_{k_{n}}p'_{l_{n}-1}-p_{k_{n}-1}p'_{l_{n}}).$$

It follows from   Lemma \ref{ss} and  (\ref{for77}) that
{\setlength{\arraycolsep}{0pt}
\begin{eqnarray} \label{for67}
&&|L_{1}(\phi_n)|\\
&=&|\alpha \alpha'(q_{k_{n}}q'_{l_{n}-1}-q_{k_{n}-1}q'_{l_{n}})-\alpha(q_{k_{n}}p'_{l_{n}-1}-q_{k_{n}-1}p'_{l_{n}})\nonumber \\
&-&\alpha'(p_{k_{n}}q'_{l_{n}-1}-p_{k_{n}-1}q'_{l_{n}})+(p_{k_{n}}p'_{l_{n}-1}-p_{k_{n}-1}p'_{l_{n}})|\nonumber\\
&=&|\alpha \alpha'(q_{k_{n}+m_n}q'_{l_{n}+m_n-1}-q_{k_{n}+m_n-1}q'_{l_{n}+m_n})\nonumber\\
&-&\alpha(q_{k_{n}+m_n}p'_{l_{n}+m_n-1}-q_{k_{n}+m_n-1}p'_{l_{n}+m_n}) \nonumber\\
&-&\alpha'(p_{k_{n}+m_n}q'_{l_{n}+m_n-1}-p_{k_{n}+m_n-1}q'_{l_{n}+m_n})\nonumber\\
&+&(p_{k_{n}+m_n}p'_{l_{n}+m_n-1}-p_{k_{n}+m_n-1}p'_{l_{n}+m_n})|\nonumber\\
&=&|(\alpha q_{k_{n}+m_n}-p_{k_{n}+m_n})(\alpha' q'_{l_{n}+m_n-1}-p'_{l_{n}+m_n-1})\nonumber \\
&-&(\alpha q_{k_{n}+m_n-1}-p_{k_{n}+m_n-1})(\alpha' q'_{l_{n}+m_n}-p'_{l_{n}+m_n})|\nonumber\\
&<&2q^{-1}_{k_{n}+m_n}q'^{-1}_{l_{n}+m_n}.\nonumber
\end{eqnarray}
}

Set $$M=\max_{n\geq1}\{(q_{n}q'_{n})^{1/n}\}.$$

Then by Lemma \ref{ss} and (iii)  of Condition \ref{c1}

{\setlength{\arraycolsep}{0pt}
\begin{eqnarray*}\label{for90}
\prod_{1\leq i\leq4}|L_{i}(\phi_n)|&\ll& q_{k_{n}}q'_{l_{n}}q^{-1}_{k_{n}+m_n}q'^{-1}_{l_{n}+m_n}\\
&\ll&2^{-\tfrac{m_n}{2}}\\
&\ll&M^{-Lm_n\delta}\\
&\ll&(q_{k_{n}}q'_{l_{n}})^{-\delta}
\end{eqnarray*}
} where $\delta=\tfrac{\log2}{2L\log M}$. Here and throughout, the
constants implied in $\ll$ depend only on $\alpha$ and $\alpha'$.

Now applying Theorem \ref{main8} implies
 that there exist a nonzero element
$(x_{1},x_{2},x_{3},x_{4})\in \mathbb{Q}^{4}$, and an infinite subset
$\mathbb{N}'$ of $\mathbb{N}$ such that

{\setlength{\arraycolsep}{0pt}
\begin{eqnarray}\label{for90}
&&x_{1}(q_{k_{n}}q'_{l_{n}-1}-q_{k_{n}-1}q'_{l_{n}})+x_{2}(q_{k_{n}}p'_{l_{n}-1}-q_{k_{n}-1}p'_{l_{n}})\\
&+&x_{3}(p_{k_{n}}q'_{l_{n}-1}-p_{k_{n}-1}q'_{l_{n}})+x_{4}(p_{k_{n}}p'_{l_{n}-1}-p_{k_{n}-1}p'_{l_{n}})\nonumber\\
&=&0 \nonumber
\end{eqnarray}
} for each $n\in \mathbb{N}'$.

\bigskip

The rest proof is divided into 3 cases.

\bigskip

 Case 1: there exists an infinite subset
$\mathbb{N}''$ of $\mathbb{N}'$ such that both the sequences $\{k_n\}_{n\in \mathbb{N}''}$, $\{l_n\}_{n\in \mathbb{N}''}$ are bounded. Then $a= a_1a_2\cdots$  and
$\textrm{a}'= a'_1a'_2\cdots$
 have the same tail.

  Case 2: there exists an infinite subset
$\mathbb{N}''$ of $\mathbb{N}'$ such that only one of the
sequences $\{k_n\}_{n\in \mathbb{N}''}$, $\{l_n\}_{n\in
\mathbb{N}''}$ is bounded.  Without loss of generality, we assume
that  $\{k_n\}_{n\in \mathbb{N}''}$ is bounded and $\{l_n\}_{n\in
\mathbb{N}''}$ is unbounded. Then there exists an infinite subset
$\mathbb{N}'''$ of $\mathbb{N}''$ and a positive integer $k$ such
that $k_n=k$ for each $n\in \mathbb{N}'''$ and $\{l_n\}_{n\in
\mathbb{N}''}$ tends to infinity. If $\alpha'$ is a quadratic
irrational number, then the continued fraction expansion
$$\alpha'=[[\alpha'];a'_1,a'_2,\cdots]$$ is eventually periodic,
and we can replace $A'_n$ with a prefix of  bounded length without
violating (i),(ii) and (iii)  of Condition \ref{c1}, and the proof
can be reduced to Case 1. If $\alpha'$ is an algebraic number of
degree higher than 2, then by (\ref{for90}) we have
{\setlength{\arraycolsep}{0pt}
\begin{eqnarray}\label{for91}
&&(x_{1}q_{k}+x_{3}p_{k})q'_{l_{n}-1}-(x_{1}q_{k-1}+x_{3}p_{k-1})q'_{l_{n}}\\
&+&(x_{2}q_{k}+x_{4}p_{k})p'_{l_{n}-1}-(x_{2}q_{k-1}+x_{4}p_{k-1})p'_{l_{n}}\nonumber\\
&=&0 \nonumber
\end{eqnarray}
}
for each  $n\in \mathbb{N}'''$. This contradicts Lemma
\ref{ssTte} since the matrix

\[\begin{matrix}\begin{pmatrix}p_{k}&p_{k-1}\\q_{k}&q_{k-1}\end{pmatrix}
\end{matrix}\] is nonsingular and hence
{\setlength{\arraycolsep}{0pt}
\begin{eqnarray*}
&&(x_{1}q_{k}+x_{3}p_{k},x_{1}q_{k-1}+x_{3}p_{k-1},x_{2}q_{k}+x_{4}p_{k},x_{2}q_{k-1}+x_{4}p_{k-1})\\
&\neq &(0,0,0,0). \\
\end{eqnarray*}
}
 Case 3: there exists an infinite subset
$\mathbb{N}''$ of $\mathbb{N}'$ such that both the sequences
$\{k_n\}_{n\in \mathbb{N}''}$, $\{l_n\}_{n\in \mathbb{N}''}$ are
strictly increasing. If at least one  of $\alpha$ and $\alpha'$ is
a quadratic irrational number, then the proof can be  reduced to
Case 2 as before. Hence we assume that both $\alpha$ and $\alpha'$
are algebraic numbers of degree higher than 2.

Set
\begin{equation}\label{foredwrf}\begin{pmatrix}p_{k_{n}}&p_{k_{n}-1}\\q_{k_{n}}&q_{k_{n}-1}\end{pmatrix}
\begin{pmatrix}q'_{l_{n}-1}&-p'_{l_{n}-1}\\-q'_{l_{n}}&p'_{l_{n}}\end{pmatrix}=\begin{pmatrix}c_{n}&-d_{n}\\ a_{n}&-b_{n}\end{pmatrix}
.\end{equation} By Lemma \ref{ss1}, we have
\begin{equation}\label{for946}\begin{vmatrix}c_{n}&-d_{n}\\ a_{n}&-b_{n}\end{vmatrix}
=\pm1.\end{equation}

Set $M_{n}=\max (|a_{n}|, |b_{n}|,|c_{n}|,|d_{n}|).$

 \begin{claim}\label{5645}
$$\lim_{n\in \mathbb{N}'' \atop n\rightarrow\infty}
M_{n}=+\infty.$$
\end{claim}

\begin{proof}
If the Claim is invalid, we can choose an infinite subset
$\mathbb{N}'''$ of $\mathbb{N}''$ such that
\[\begin{matrix}\label{for89}\begin{pmatrix}c_{n}&-d_{n}\\ a_{n}&-b_{n}\end{pmatrix}=\begin{pmatrix}c&-d\\
a&-b\end{pmatrix}\\
\end{matrix}\] is a constant matrix, when $n\in \mathbb{N}'''$. By (\ref{foredwrf}), this
means that
$$\textsl{M}(a_{k_{n}+1},\cdots,a_{k_{n+1}})=\textsl{M}(a'_{l_{n}+1},\cdots,a'_{l_{n+1}}),
$$ for each $n\in \mathbb{N}'''$.
Applying Lemma \ref{ss3}, we get
$$a_{k_{n}+1}\cdots a_{k_{n+1}}=a'_{l_{n}+1} \cdots
a'_{l_{n+1}},$$ for each $n\in \mathbb{N}'''$.  But this
contradicts Convention \ref{convi1}
\end{proof}

Now we assume that $x_{1}\neq0$, the other three cases can be
reduced to the case $x_{1}\neq0$ by replacing $\alpha$  with
$1/\alpha$ and/or replacing $\alpha'$  with $1/\alpha'$.  We can
further assume that $x_{1}=-1$ without loss of generality.

 Hence by
(\ref{for90}) we have

\begin{equation}\label{for8y}a_{n}=x_{2}b_{n}+x_{3}c_{n}+x_{4}d_{n},\end{equation} for $n\in \mathbb{N}''$.

From now on  we assume that there exists an infinite subset
$\mathbb{N}'''$ of $\mathbb{N}''$¡¡¡¡¡¡such that
\begin{equation}\label{for9y}q_{k_{n}}\leq
q'_{l_{n}}\end{equation} for each $n\in \mathbb{N}'''$;  the other
case can be treated similarly.

Now we define three linearly independent linear forms as follows:
{\setlength{\arraycolsep}{0pt}
\begin{eqnarray*}
L'_{1}(Y_1,Y_2,Y_3)&=&L_{1}(x_{2}Y_1+x_{3}Y_2+x_{4}Y_3,Y_1,Y_2,Y_3)\\
&=&(\alpha \alpha'x_2-\alpha )Y_1+(\alpha \alpha'x_3-\alpha' )Y_2+(\alpha \alpha'x_4+1)Y_3,\\
L'_{2}(Y_1,Y_2,Y_3)&=&\alpha Y_1-Y_3,\\
L'_{3}(Y_1,Y_2,Y_3)&=&Y_3.\\
\end{eqnarray*}
} (If $q_{k_{n}}\geq q'_{l_{n}}$ for infinitely many $n$, we can
set $L'_{2}(Y_1,Y_2,Y_3)=\alpha' Y_2-Y_3.$) Then by (\ref{for67})
and (\ref{for9y}), we have

{\setlength{\arraycolsep}{0pt}
\begin{eqnarray}\label{for900}
&& \prod_{1\leq i\leq3}|L'_{i}(b_{n},c_{n},d_{n})|\\
&\ll& q_{k_{n}}q'_{l_{n}}q^{-1}_{k_{n}+m_n}q'^{-1}_{l_{n}+m_n}\nonumber \\
&\ll&(q_{k_{n}}q'_{l_{n}})^{-\delta} \nonumber
\end{eqnarray}
} for $n\in \mathbb{N}'''$.

Now applying Theorem \ref{main8} implies
 that there exist a nonzero element
$(y_{1},y_{2},y_{3})\in \mathbb{Q}^{3}$, and an infinite subset
$\mathbb{N}^{(4)}$ of $\mathbb{N'''}$ such that
\begin{equation}\label{for82t}y_{1}b_{n}+y_{2}c_{n}+y_{3}d_{n}=0,\end{equation}
for each $n\in \mathbb{N}^{(4)}$.

 \begin{claim}\label{56451}
$$(y_{1},y_{2})\neq(0,0),$$ and $d_{n}\neq0$ if $n\in \mathbb{N}^{(4)}$ is sufficiently large.
\end{claim}

\begin{proof}
If $y_{1}=y_{2}=0$, we have $d_{n}=0$ for  $n\in \mathbb{N}'''$.
This force
\begin{equation}\label{for82rt}|b_{n}|=|c_{n}|=1\end{equation}
for  $n\in \mathbb{N}^{(4)}$. Combining Claim \ref{5645},
(\ref{for8y}) and (\ref{for82rt}) immediately yields a
contradiction.
\end{proof}

We assume that $y_{1}\neq0$; the case  $y_{2}\neq0$ can be treated
similarly. We can further assume that $y_{1}=-1$ without loss of
generality. Hence
\begin{equation}\label{for82ty}b_{n}=y_{2}c_{n}+y_{3}d_{n},\end{equation}
for each $n\in \mathbb{N}^{(4)}$. Set
{\setlength{\arraycolsep}{0pt}
\begin{eqnarray*}
L''_{1}(Z_1,Z_2)&=&L'_{1}(y_{2}Z_1+y_{3}Z_2,Z_1,Z_2),\\
&=&[y_{2}(\alpha \alpha'x_2-\alpha )+(\alpha \alpha'x_3-\alpha' )]Z_1,\\
&+&[y_{3}(\alpha \alpha'x_2-\alpha )+(\alpha \alpha'x_4+1)]Z_2.\\
\end{eqnarray*}
}

First we point out that $L''_{1}(Z_1,Z_2)\neq0$, otherwise, we
would have
$$\alpha'=\tfrac{y_{2}\alpha}{(y_{2}x_2+x_3)\alpha-1}=\tfrac{y_{3}\alpha-1}{\alpha(y_{3}x_2+x_4)},$$
which contradicts the fact that  $\alpha$ is of degree higher than
2.

 By (\ref{for67}) we have
$$|L''_{1}(c_{n},d_{n})|\ll(q_{k_{n}}q'_{l_{n}})^{-1-\delta}$$
Now we get as above  that there exist a  $z\in \mathbb{Q}$, and an
infinite subset $\mathbb{N}^{(5)}$ of $\mathbb{N}^{(4)}$ such that
\begin{equation}\label{f82t}c_{n}=zd_{n},\end{equation}
for each $n\in \mathbb{N}^{(5)}$. Combining (\ref{for8y}),
(\ref{for82ty})  and  (\ref{f82t}) implies that
\[\begin{matrix}\label{for94}\begin{pmatrix}c_{n}&-d_{n}\\ a_{n}&-b_{n}\end{pmatrix}=d_{n}D
\\
\end{matrix}\] for each $n\in \mathbb{N}^{(5)}$, where $D$ is a
constant matrix. By Claim  \ref{5645}, $|d_{n}|$ tends to infinity
along $\mathbb{N}^{(5)}$. This contradicts (\ref{for946}) and
finishes the proof of the first assertion of Theorem \ref{main4}.
Now assume that
$$\limsup_{n\rightarrow\infty} \mid|A_n|-|A'_n|\mid=+\infty.$$ If
one of $\alpha$ and $\alpha'$ is of degree higher than 2, then
applying the arguments in cases 2 and 3, we get a contradiction.
\end{proof}

 \section{Proof of Theorem \ref{main5}}
As the  proof is similar to that of Theorem \ref{main4}, many
details are omitted.

\begin{proof}[Proof of Theorem \ref{main5}]
Assume that (i'), (ii) and (iii) of Condition \ref{c1} is
satisfied with three sequences of finite nonempty words
$\{A_n\}_{n\geq1}$, $\{A'_n\}_{n\geq1}$, $\{B_n\}_{n\geq1}$. Set
$k_n=|A_n|$, $l_n=|A'_n|$, and $m_n=|B_n|$. Then we have

\[\begin{matrix}\begin{pmatrix}p_{k_{n}}&p_{k_{n}-1}\\q_{k_{n}}&q_{k_{n}-1}\end{pmatrix}\textsl{M}(B_n)=
\begin{pmatrix}p_{k_{n}+m_n}&p_{k_{n}+m_n-1}\\q_{k_{n}+m_n}&q_{k_{n}+m_n-1}\end{pmatrix},\\
\end{matrix}\]

and

\[\begin{matrix}\begin{pmatrix}p'_{l_{n}}&p'_{l_{n}-1}\\q'_{l_{n}}&q'_{l_{n}-1}\end{pmatrix}\textsl{M}(B_n)^T=
\begin{pmatrix}p'_{l_{n}+m_n}&p'_{l_{n}+m_n-1}\\q'_{l_{n}+m_n}&q'_{l_{n}+m_n-1}\end{pmatrix}.\\
\end{matrix}\]
\bigskip
The above two identities immediately implies
{\setlength{\arraycolsep}{0pt}
\begin{eqnarray} \label{for89'}
&&\begin{pmatrix}p_{k_{n}}&p_{k_{n}-1}\\q_{k_{n}}&q_{k_{n}-1}\end{pmatrix}
\begin{pmatrix}p'_{l_{n}+m_n}&q'_{l_{n}+m_n}\\p'_{l_{n}+m_n-1}&q'_{l_{n}+m_n-1}\end{pmatrix} \\
&=&\begin{pmatrix}p_{k_{n}+m_n}&p_{k_{n}+m_n-1}\\q_{k_{n}+m_n}&q_{k_{n}+m_n-1}\end{pmatrix}
\begin{pmatrix}p'_{l_{n}}&q'_{l_{n}}\\p'_{l_{n}-1}&q'_{l_{n}-1}\end{pmatrix}\nonumber
\\
&=&\begin{pmatrix}d_{n}&c_{n}\\ b_{n}&a_{n}\end{pmatrix}.\nonumber
\end{eqnarray}
} Evaluating the  linear forms {\setlength{\arraycolsep}{0pt}
\begin{eqnarray*}
L_{1}(X_1,X_2,X_3,X_4)&=&\alpha \alpha'X_1-\alpha X_2-\alpha'X_3+X_4,\\
L_{2}(X_1,X_2,X_3,X_4)&=&\alpha'X_1-X_2\\
L_{3}(X_1,X_2,X_3,X_4)&=&\alpha X_1-X_3 \\
L_{4}(X_1,X_2,X_3,X_4)&=&X_1
\end{eqnarray*}
} on the quadruple {\setlength{\arraycolsep}{0pt}
\begin{eqnarray*}
&&(a_{n},b_{n},c_{n},d_{n})\\
&=&(q_{k_{n}}q'_{l_{n}+m_n}+q_{k_{n}-1}q'_{l_{n}+m_n-1},q_{k_{n}}p'_{l_{n}+m_n}+q_{k_{n}-1}p'_{l_{n}+m_n-1},\\
&&p_{k_{n}}q'_{l_{n}+m_n}+p_{k_{n}-1}q'_{l_{n}+m_n-1},p_{k_{n}}p'_{l_{n}+m_n}+p_{k_{n}-1}p'_{l_{n}+m_n-1}),
\end{eqnarray*}
} we get
 {\setlength{\arraycolsep}{0pt}
\begin{eqnarray*}
|L_{1}(a_{n},b_{n},c_{n},d_{n})|&\ll&q^{-1}_{k_{n}}q'^{-1}_{l_{n}+m_n},\\
|L_{2}(a_{n},b_{n},c_{n},d_{n})|&\ll&q_{k_{n}}q'^{-1}_{l_{n}+m_n},\\
|L_{4}(a_{n},b_{n},c_{n},d_{n})|&\ll&q_{k_{n}}q'_{l_{n}+m_n}.
\end{eqnarray*}
} By  (\ref{for89'}) we have
 {\setlength{\arraycolsep}{0pt}
\begin{eqnarray*}
&&|L_{3}(a_{n},b_{n},c_{n},d_{n})|\\
&=&|\alpha(q_{k_{n}+m_n}q'_{l_{n}}+q_{k_{n}+m_n-1}q'_{l_{n}-1})-(p_{k_{n}+m_n}q'_{l_{n}}+p_{k_{n}+m_n-1}q'_{l_{n}-1})|\\
&\ll&q^{-1}_{k_{n}+m_n}q'_{l_{n}}.\\
\end{eqnarray*}
}

Now applying Theorem \ref{main8} as before implies
 that there exist a nonzero element
$(x_{1},x_{2},x_{3},x_{4})\in \mathbb{Q}^{4}$, and an infinite
subset $\mathbb{N}'$ of $\mathbb{N}$ such that

\begin{equation}\label{for8y'}x_{1}a_{n}+x_{2}b_{n}+x_{3}c_{n}+x_{4}d_{n}=0,\end{equation} for each $n\in \mathbb{N}'$.

At this point the proof is divided into 3 cases as before.

 Case 1: there exists an infinite subset
$\mathbb{N}''$ of $\mathbb{N}'$ such that both the sequences
$\{k_n\}_{n\in \mathbb{N}''}$, $\{l_n\}_{n\in \mathbb{N}''}$ are
bounded. Without loss of generality, we assume that $k_n=l_n=0$
for each $n\in \mathbb{N}''$. Then by (\ref{for89'}) and
(\ref{for8y'}),  we have {\setlength{\arraycolsep}{0pt}
\begin{eqnarray}\label{xx8y'}
&&x_{1}a_{n}+x_{2}b_{n}+x_{3}c_{n}+x_{4}d_{n}\\
&=&x_{1}p'_{m_n-1}+x_{2}q'_{m_n-1}+x_{3}p'_{m_n-1}+x_{4}q'_{m_n}\nonumber\\
&=&x_{1}p_{m_n-1}+x_{2}p_{m_n}+x_{3}q_{m_n-1}+x_{4}q_{m_n}\nonumber\\
&=&0\nonumber
\end{eqnarray}
} for each $n\in \mathbb{N}''$.

By Lemma \ref{ssTte} and (\ref{xx8y'}), both $\alpha$ and
$\alpha'$ are quadratic irrationals. Assume that the shortest
periods of the continued fraction expansions of $\alpha$ and
$\alpha'$ are $A$ and $B$ respectively. Then  (i') of Condition
\ref{c1} implies that $A$ and $B$ are cycle   mirror symmetry.

 Case 2: there exists an infinite subset
$\mathbb{N}''$ of $\mathbb{N}'$ such that only one of the
sequences $\{k_n\}_{n\in \mathbb{N}''}$, $\{l_n\}_{n\in
\mathbb{N}''}$ is bounded.  Without loss of generality, we assume
that $\{k_n\}_{n\in \mathbb{N}''}$ is bounded. Then there exists
an infinite subset $\mathbb{N}'''$ of $\mathbb{N}''$ and a
nonnegative integer $k$ such that $k_n=k$ for each $n\in
\mathbb{N}'''$. If $\alpha'$ is a quadratic irrational number,
then the continued fraction expansion
$$\alpha'=[[\alpha'];a'_1,a'_2,\cdots]$$ is eventually periodic,
and we can replace $A'_n$ by a prefix of  bounded length  without
violating (i'), (ii) and (iii) of Condition \ref{c1}, and the
proof can be reduced to Case 1. If $\alpha'$ is an algebraic
number of degree higher than 2, then by (\ref{for8y'}) we have
{\setlength{\arraycolsep}{0pt}
\begin{eqnarray}\label{for91}
&&(x_{1}q_{k}+x_{3}p_{k})q'_{l_{n}+m_n-1}-(x_{1}q_{k-1}+x_{3}p_{k-1})q'_{l_{n}+m_n}\\
&+&(x_{2}q_{k}+x_{4}p_{k})p'_{l_{n}+m_n-1}-(x_{2}q_{k-1}+x_{4}p_{k-1})p'_{l_{n}+m_n}\nonumber\\
&=&0 \nonumber
\end{eqnarray}
} for each  $n\in \mathbb{N}'''$. This contradicts Lemma
\ref{ssTte}.

 Case 3: there exists an infinite subset
$\mathbb{N}''$ of $\mathbb{N}'$ such that both the sequences
$\{k_n\}_{n\in \mathbb{N}''}$, $\{l_n\}_{n\in \mathbb{N}''}$ are
strictly increasing. If at least one  of $\alpha$ and $\alpha'$ is
a quadratic irrational number, then the proof can be  reduced to
Case 2. Hence we assume that both $\alpha$ and $\alpha'$ are
algebraic numbers of degree higher than 2. Without loss of
generality we can further assume  that $\alpha,\alpha'\in(0,1)$.
Hence we have $p_{n},p'_{n}>0$ for each $n$. Set $$M_{n}=\max
(|a_{n}|, |b_{n}|,|c_{n}|,|d_{n}|).$$ Then it is easy to see that
$$\lim_{n\rightarrow\infty}
M_{n}=+\infty.$$

Now the rest proof proceeds in the same way as in Case 3 of the
proof of Theorem \ref{main4}.
\end{proof}

\section{proofs of Theorems
\ref{main99} and \ref{main990}} This section is devoted to the
proofs of Theorems \ref{main99} and \ref{main990}.

First we need two auxiliary lemmas, the first of which follows
directly from \cite[Lemma 2.2]{bu1} and its proof.
  \begin{lemma}\label{ss091}
Let $$\alpha=[a_0;a_1,a_2,\cdots],$$ and
$$\beta=[b_0;b_1,b_2,\cdots]$$ be  the continued fraction expansions of two
real numbers, and let $\{\frac{p_{n}}{q_{n}}\}_{n\geq0}$ the
sequence of convergents of $\beta$. Let $n$ be a nonnegative
integer such that $a_i=b_i$ for $i=1,\cdots,n-1$, and $a_{n}\neq
b_{n}$. Then we have $$|\alpha-\beta|\geq
\frac{1}{72q^2_{n}b_{n+1}b_{n+2}}\geq \frac{1}{72q_{n}q_{n+2}} .$$
   \end{lemma}
   \begin{lemma}\label{ss09}
Let $$\xi=[a_0;a_1,a_2,\cdots,a_n,\cdots]$$ be the continued
fractional expansion of an irrational real algebraic number, and
let
 $\{p_{n}/q_{n}\}_{n\geq0}$ be the sequence of convergents. Let
$k$ be a positive integer and let $\epsilon$ be a positive number.
Then there can not exist an infinite subset $\mathbb{N}'$ of
$\mathbb{N}$ such that $$q_{n+k}>q_{n}^{1+\epsilon},$$ for each
$n\in \mathbb{N}'$.
   \end{lemma}
\begin{proof}Assume that there exists an infinite subset $\mathbb{N}'$ of
$\mathbb{N}$ such that $$q_{n+k}>q_{n}^{1+\epsilon},$$ for each
$n\in \mathbb{N}'$. Set $1+\varepsilon=\sqrt[k]{1+\epsilon}$. Then
for any  $n\in \mathbb{N}'$, there exists a $0\leq i_n<k $ such
that  $$q_{n+i_n+1}>q_{n+i_n}^{\varepsilon}.$$ Now  we have
$$|\xi-\frac{p_{n+i_n}}{q_{n+i_n}}|<\frac{1}{q_{n+i_n}q_{n+i_n+1}}<\frac{1}{q_{n+i_n}^{2+\varepsilon}},$$ for each  $n\in
\mathbb{N}'$.
This contradicts Roth's theorem \cite{ro}.
\end{proof}

\begin{proof}[Proof of Theorem \ref{main99}]
The proof is divided into several claims. Assume that  there exist
a sequence $\{\beta_{n}\}_{n\geq0}$ of distinct elements from
$\Theta_{\alpha}$ such that
\begin{equation}|\xi-\beta_{n}|<\|\beta_{n}\|^{-1-\epsilon}.\end{equation}
\bigskip
 \begin{claim}\label{zawdremvf}
$$\lim_{n\rightarrow\infty}
\|\beta_{n}\|=+\infty,$$ and $$\lim_{n\rightarrow\infty}
\beta_{n}=\xi.$$
\end{claim}

\begin{proof}
Otherwise, we can choose an infinite subset
$\mathbb{N}'$ of $\mathbb{N}$ such that  \[\begin{matrix}\beta_{n}=\begin{pmatrix}a_n&b_n\\
c&d\end{pmatrix}\alpha,
\\
\end{matrix}\] and the determinent   \[\begin{matrix}\begin{vmatrix}a_n&b_n\\
c&d\end{vmatrix},
\end{matrix}\] is fixed when  $n\in \mathbb{N}'$£¬ where $c$ and $d$ are constants. For distinct $\beta_{n_1} $ and $\beta_{n_2}$, we have
  \[\begin{matrix} |\beta_{n_1}-\beta_{n_2}|=|\begin{pmatrix}a_{n_1}-a_{n_2}&b_{n_1}-b_{n_2}\\
c&d\end{pmatrix}\alpha|.
\\
\end{matrix}\] We note that $c$ and $d$ are co-prime and   \[\begin{matrix} \begin{vmatrix}a_{n_1}-a_{n_2}&b_{n_1}-b_{n_2}\\
c&d\end{vmatrix}=0.
\\
\end{matrix}\]  Hence  \[\begin{matrix} \begin{pmatrix}a_{n_1}-a_{n_2}&b_{n_1}-b_{n_2}\\
c&d\end{pmatrix}\alpha
\\
\end{matrix}\] is a nonzero integer and $|\beta_{n_1}-\beta_{n_2}|\geq1$. This contradicts the fact that \begin{equation}|\xi-\beta_{n}|<\|\beta_{n}\|^{-1-\epsilon}\end{equation} and that $\{\beta_{n}\}_{n\geq0}$ consists of distinct elements.
\end{proof}

From now on, we assume without loss of generality that
$\alpha,\beta_{n}$ and $\xi$ all lie in the interval $(0,1)$. Now
by Lemma \ref{xxx}, we can assume that
$$\alpha=[0,a_1,a_2,\cdots,a_{k_n-1},a_{k_n},a_{k_n+1},\cdots ],$$
and
$$\beta_{n}=[0,a^{(n)}_1,a^{(n)}_2,\cdots,a^{(n)}_{l_n},a_{k_n+1},a_{k_n+2},\cdots ],$$ where $k_n,l_n\geq0$ and $a^{(n)}_{l_n}\neq a_{k_n}$.

Let $\{p_{k}/q_{k}\}_{k\geq1}$ and
$\{p^{(n)}_{k}/q^{(n)}_{k}\}_{k\geq1}$ be  respectively the
sequence of convergents of $\alpha$ and $\beta_{n}$. Then we have

\begin{equation}\label{cdremvf}\beta_{n}=\frac{(p^{(n)}_{l_n}q_{k_{n}-1}-p^{(n)}_{l_n-1}q_{k_{n}})\alpha-(p^{(n)}_{l_n}p_{k_{n}-1}-p^{(n)}_{l_n-1}p_{k_{n}})}
{(q^{(n)}_{l_n}q_{k_{n}-1}-q^{(n)}_{l_n-1}q_{k_{n}})\alpha-(q^{(n)}_{l_n}p_{k_{n}-1}-q^{(n)}_{l_n-1}p_{k_{n}})}.\end{equation}

 \begin{claim}
$$\lim_{n\rightarrow\infty}
l_n+k_{n}=+\infty.$$
\end{claim}
\begin{proof}
Otherwise there would exist an infinite subset $\mathbb{N}'$ of
$\mathbb{N}$ and two constants $k$ and $l$  such that $l_n=l$ and
$k_{n}=k$ for each $n\in \mathbb{N}'$. As
$$\lim_{n\rightarrow\infty} \beta_{n}=\xi,$$
$a^{(n)}_1,a^{(n)}_2,\cdots,a^{(n)}_{l}$ will be fixed when $n\in
\mathbb{N}'$ is sufficiently large. This implies
$\xi\in\Theta_{\alpha}$ which contradicts our assumption.
\end{proof}

Let $\epsilon_{1}$ be another positive number.
 \begin{claim}\label{cdrdeem1111}
There exists an infinite subset $\mathbb{N}'$ of $\mathbb{N}$ such
that $\|\beta_{n}\|\geq
(\overline{q^{(n)}_{l_n-1}}q_{k_{n}})^{1-\epsilon_{1}}$ for each
$n\in \mathbb{N}'$, where $\overline{x}=\max(1,x).$
\end{claim}

\begin{proof} The proof is divided into three cases.

Case 1, there exists an infinite subset $\mathbb{N}'$ of
$\mathbb{N}$ such that $k_{n}=k$ is a constant for $n\in
\mathbb{N}'$. Then by (\ref{cdremvf}) there exist a positive
constant $M$ such that $$\|\beta_{n}\|\geq M q^{(n)}_{l_n-1}\geq
(\overline{q^{(n)}_{l_n-1}}q_{k_{n}})^{1-\epsilon_{1}}$$ when
$n\in \mathbb{N}'$ is sufficiently large.

Case 2, there exists an infinite subset $\mathbb{N}'$ of
$\mathbb{N}$ such that  $l_n=0$ for each $n\in \mathbb{N}'$. Then
by (\ref{cdremvf}) we have
$$\|\beta_{n}\|\geq q_{k_{n}-1}$$
when $n\in \mathbb{N}'$.
On the other hand, by Lemma \ref{ss09}, we have
$$q_{k_{n}-1}\geq
(\overline{q^{(n)}_{l_n-1}}q_{k_{n}})^{1-\epsilon_{1}},$$ when
$n\in \mathbb{N}'$ is sufficiently large.

Case 3, there exists an infinite subset $\mathbb{N}'$ of
$\mathbb{N}$ such that the sequence $\{k_{n}\}_{n\in \mathbb{N}'}$
is strictly increasing and $l_n>0$ for each $n\in \mathbb{N}'$.
Then by (\ref{cdremvf}), it suffices to show that
$$|\tfrac{q^{(n)}_{l_n}q_{k_{n}-1}}{q^{(n)}_{l_n-1}q_{k_{n}}}-1|>(q^{(n)}_{l_n-1}q_{k_{n}})^{-\epsilon_{1}}.$$
By (\ref{ss12}),
$\frac{q^{(n)}_{l_n}q_{k_{n}-1}}{q^{(n)}_{l_n-1}q_{k_{n}}}$ is the
quotient of  the two continued fractions
$\overline{\beta_{n}}=[a^{(n)}_{l_n};a^{(n)}_{l_n-1},\cdots,a^{(n)}_1]$
and $\overline{\alpha}=[a_{k_n};a_{k_n-1},\cdots,a_1]$.  By Lemma
\ref{ss09}, we have
\begin{equation}\label{cdrdeemcsz}\tfrac{q_{k_{n}-3}}{q_{k_{n}}}\geq q_{k_{n}}^{-\epsilon_{1}},\end{equation}
when $n\in \mathbb{N}''$ is sufficiently large. Now it follows
from Lemma \ref{ss091} and (\ref{cdrdeemcsz}) that
{\setlength{\arraycolsep}{0pt}
\begin{eqnarray} \label{fofdhr67}
&&|\tfrac{q^{(n)}_{l_n}q_{k_{n}-1}}{q^{(n)}_{l_n-1}q_{k_{n}}}-1|\\
&=&\frac{1}{\overline{\alpha}}|\overline{\alpha}-\overline{\beta_{n}}|\nonumber \\
&\geq&\frac{q_{k_{n}-1}}{72q_{k_{n}}a_{k_n-1}a_{k_n-2}}\nonumber\\
&\geq&\frac{q_{k_{n}-3}}{72q_{k_{n}}}\geq q_{k_{n}}^{-\epsilon_{1}}\nonumber\\
&\geq&(q^{(n)}_{l_n-1}q_{k_{n}})^{-\epsilon_{1}},\nonumber
\end{eqnarray}
} when $n\in \mathbb{N}''$ is sufficiently large.
\end{proof}

Fix an infinite subset $\mathbb{N}'$ of $\mathbb{N}$ satisfying
Claim \ref{cdrdeem1111}. Let
$$\xi=[0,b_1,b_2,b_{3},\cdots ]$$ and let
$\{p'_{k}/q'_{k}\}_{k\geq1}$  be the sequence of convergents.

 \begin{claim}\label{cdrdeem11}
When $n\in \mathbb{N}'$ is sufficiently large, we
have $b_{s}=a^{(n)}_s$ for $1\leq s\leq l_n$. Moreover, when $n\in
\mathbb{N}'$ is sufficiently large, let $m_n$ be the nonnegative
integer such that $b_{l_n+s}=a_{k_n+s}$ for $0\leq s\leq m_n$ and
$b_{l_n+m_n+1}\neq a_{k_n+m_n+1}$. Then
$$\lim_{n\rightarrow\infty} m_n=+\infty.$$
\end{claim}

\begin{proof}Assume that there exists a positive integer $t_n< l_n$ such that $b_{s}=a^{(n)}_s$ for $1\leq s< t_n$ and $b_{t_n}\neq
a^{(n)}_{t_n}$. Then by Claim \ref{zawdremvf} we have
$$\lim_{n\rightarrow\infty} t_n=+\infty.$$ Now by Lemma \ref{ss091}
and Claim \ref{cdrdeem1111},
\begin{equation}\label{cdrdeem}
\frac{1}{72q'_{t_n}q'_{t_n+2}}\leq|\xi-\beta_{n}|\leq\|\beta_{n}\|^{-2-\epsilon}\leq
(q^{(n)}_{l_n-1}q_{k_{n}})^{(-2-\epsilon)(1-\epsilon_{1})}\leq
q'^{(-2-\epsilon)(1-\epsilon_{1})}_{t_n-1},\end{equation}  when
$n\in \mathbb{N}'$. We can choose $\epsilon_{1}$ such that
$(-2-\epsilon)(1-\epsilon_{1})<-2$. But then, by Lemma \ref{ss09},
(\ref{cdrdeem}) is impossible when $n$ is sufficiently large. The
same proof shows that
$$\lim_{n\rightarrow\infty} m_n=+\infty.$$
 \end{proof}

From now on, we always assume that $n$ lies in $\mathbb{N}'$ and
is sufficiently large.

 \begin{claim}\label{cdrdeem11fre}There exist two positive numbers $\delta$ and
 $L$ such that
$$(q_{k_{n}}q'_{l_{n}})^{1+\delta}< L q_{k_{n}+m_n}q'_{l_{n}+m_n}.$$
\end{claim}
\begin{proof}
Set
\[\begin{matrix}\textsl{M}(b_{l_n+1}
\cdots b_{l_n+m_n})=\textsl{M}(a_{k_n+1} \cdots a_{k_n+m_n})=\begin{pmatrix}p''_{m_n}&p''_{m_n-1}\\q''_{m_n}&q''_{m_n-1}\end{pmatrix}.\\
\end{matrix}\]
Then we have
 $$q'_{l_{n}+m_n}=q'_{l_{n}}p''_{m_n}+q'_{l_{n}-1}q''_{m_n},$$
$$q_{k_{n}+m_n}=q_{k_{n}}p''_{m_n}+q_{k_{n}-1}q''_{m_n},$$
and $$q''_{m_n}\leq p''_{m_n}.$$ Hence
\begin{equation}\label{cdrdeemzrea}q'_{l_{n}+m_n}\leq q'_{l_{n}}(p''_{m_n}+q''_{m_n})\leq
\frac{ 2q'_{l_{n}}q_{k_{n}+m_n}}{q_{k_{n}}}.\end{equation}

By Lemma \ref{ss091} and Claim \ref{cdrdeem1111}, we have
\begin{equation}\label{cdrdeemza}
\frac{1}{72q'_{l_n+m_n+1}q'_{l_n+m_n+3}}\leq|\xi-\beta_{n}|\leq
(\overline{q^{(n)}_{l_n-1}}q_{k_{n}})^{(-2-\epsilon)}.\end{equation}
\bigskip

 Lemma \ref{ss09} implies that for any small positive integer $\epsilon_{2}$, there exists a positive number $M$
 such that
\begin{equation}\label{cdrdeemzaax}
M q'^{2+\epsilon_{2}}_{l_n+m_n}\geq
q'_{l_n+m_n+1}q'_{l_n+m_n+3},\end{equation}and
\begin{equation}\label{cdrdeemzaax1}
M \overline{q^{(n)}_{l_n-1}}^{1+\epsilon_{2}}\geq
q'_{l_n}.\end{equation}

Combining (\ref{cdrdeemzrea}), (\ref{cdrdeemza}),
(\ref{cdrdeemzaax}) and (\ref{cdrdeemzaax1})
 and choosing $\epsilon_{1}$ and
$\epsilon_{2}$ small enough imply the Claim.
 \end{proof}

We are now in the position to prove Theorem \ref{main99}. Set
$$A_n=0 a_1 \cdots a_{k_n},$$ $$A'_n=0 b_1 b_2\cdots
b_{l_n},$$ and
$$B_n=b_{l_n+1} \cdots b_{l_n+m_n}=a_{k_n+1} \cdots a_{k_n+m_n}.$$
Then by Claims \ref{cdrdeem11} and \ref{cdrdeem11fre}, the three
sequences $\{A_n\}_{n\geq1}$, $\{A'_n\}_{n\geq1}$,
$\{B_n\}_{n\geq1}$ satisfy Condition \ref{cfggh1}, and (i) and
(ii) of Condition \ref{c1}. Now applying Theorem \ref{main5} and
Remark \ref{remar} implies that $\xi$ is transcendental and
finishes the proof.
\end{proof}

\begin{proof}[Proof of Theorem \ref{main990}]
Assume that  there exist a sequence $\{\beta_{n}\}_{n\geq0}$ of
distinct elements from $\Theta_{\alpha}$ such that
\begin{equation}|\xi-\beta_{n}|<\|\beta_{n}\|^{-1-\epsilon}.\end{equation}
Then the above proof and notations can be directly applied. In the
quadratic case, by replacing $\{\beta_{n}\}_{n\geq0}$ with a
subsequence, we can assume that $k_{n}=0$,
$$\alpha=[0,\overline{a_1,a_2,\cdots,a_{k}}],$$ and $$\beta_{n}=[0,a^{(n)}_1,a^{(n)}_2,\cdots,a^{(n)}_{l_n},\overline{a_1,a_2,\cdots,a_{k}}],$$
where $a^{(n)}_{l_n}\neq a_{k}$. Hence
\begin{equation}\label{cdremvfdffc}\beta_{n}=\tfrac{p^{(n)}_{l_n-1}\alpha+p^{(n)}_{l_n}}
{q^{(n)}_{l_n-1}\alpha+q^{(n)}_{l_n}},\end{equation} and
\begin{equation}\label{cdremvfdffc}\|\beta_{n}\|=
|\tfrac{(q^{(n)}_{l_n-1}\alpha+q^{(n)}_{l_n})(q^{(n)}_{l_n-1}\alpha^{\sigma}+q^{(n)}_{l_n})}{\alpha-\alpha^{\sigma}}|.\end{equation}

It is well-known that the Galois conjugate of $\alpha$ is
$$\alpha^{\sigma}=-[a_{k};\overline{a_{k-1},\cdots,a_1,a_{k}}].$$
Hence by Lemma \ref{ss091}, we have {\setlength{\arraycolsep}{0pt}
\begin{eqnarray} \label{for67xds}
&&|q^{(n)}_{l_n-1}\alpha^{\sigma}+q^{(n)}_{l_n})|\\
&=&q^{(n)}_{l_n-1}|[a_{k};\overline{a_{k-1},\cdots,a_1,a_{k}}]-[a^{(n)}_{l_n};\cdots,a^{(n)}_1]|\nonumber \\
&\geq&\tfrac{q^{(n)}_{l_n-1}}{72a_{k-1}a_{k-2}}.\nonumber
\end{eqnarray}
} Hence there exists a positive constant $M$ (depends on $\alpha$)
such that
\begin{equation}\label{cdremvfdffcwd}\|\beta_{n}\|\geq Mq^{(n)}_{l_n-1}q^{(n)}_{l_n}.\end{equation}

Now the rest proof proceeds exactly as before.
\end{proof}

We close this paper with a question. In \cite{ab1}, Theorem
\ref{main3} was generalized to the case of several irrational real
algebraic numbers. What can we say about
 the continued fraction expansions of several
irrational real algebraic numbers?


\begin{thebibliography}{99}
\bibitem{ab2} B. Adamczewski and Y. Bugeaud, \textsl{On the complexity of algebraic
numbers. II. Continued fractions}, Acta Math. 195 (2005), 1-20.


\bibitem{ab} B. Adamczewski and Y. Bugeaud, \textsl{On the complexity of algebraic
numbers I, Expansions in integer bases}, Ann. Math. 165 (2007),
547--566.

\bibitem{ab1} B. Adamczewski and Y. Bugeaud, \textsl{On the independence of expansions
of algebraic numbers in an integer base}, Bull. Lond. Math. Soc.
39 (2007), 283--289.

\bibitem{ab3} B. Adamczewski and Y. Bugeaud, \textsl{On the Maillet-Baker continued fractions}, J. reine angew. Math. 606 (2007),
105-121.

\bibitem{ab4} B. Adamczewski and Y. Bugeaud, \textsl{Palindromic continued fractions}, Ann. Inst. Fourier (Grenoble) 57
(2007), 1557-1574.

\bibitem{abd} B. Adamczewski, Y. Bugeaud and L. Davison, \textsl{Continued fractions and transcendental numbers}, Ann Inst. Fourier (Grenoble) 56 (2006), 2093-2113.



\bibitem{bu}Y. Bugeaud, \textsl{Automatic continued fractions are
transcendental or quadratic}, Ann. Sci. \'{e}c. Norm. Sup\'{e}r.
(4) 46 (2013), no. 6, 1005-1022.

\bibitem{bu1}Y. Bugeaud, \textsl{On the quadratic Lagrange spectrum}, Math. Z. 276(3-4), 985-999 (2014).

\bibitem{co}A. Cobham, \textsl{Uniform tag sequences}, Math. Systems Theory 6 (1972), 164-192.


\bibitem{fr} J. S. Frame, \textsl{Continued Fractions and Matrices}, Amer.
Math. Monthly 56 (1949), no. 2, 98-103.

\bibitem{hw} G. Hardy and E. Wright, An introduction to the theory of
numbers, Oxford Univ. Press, London, 1979.

\bibitem{li}X. Lin, \textsl{ $b$-ary expansions of algebraic numbers},
preprint.

\bibitem{ro}K. F. Roth, \textsl{Rational approximations to algebraic numbers},
Mathematika 2 (1955), 1-20.

\bibitem{sch}W. M. Schmidt, \textsl{Diophantine approximation}, Lecture Notes in Mathematics 785 (Springer, 1980).


\end{thebibliography}
\end{document}